\documentclass[11pt]{article}
\usepackage{amsfonts}
\usepackage{changepage}
\usepackage{amsmath,amssymb,amsthm,amsfonts,enumerate}

\usepackage{graphicx}
\usepackage{float}
\usepackage[scriptsize,nooneline]{subfigure}
\usepackage{indentfirst}
\usepackage{cite}
\usepackage{color}
\usepackage{mathrsfs}
\usepackage{epstopdf}
\usepackage{setspace}
\usepackage[labelfont=footnotesize,font=footnotesize]{caption}
\usepackage[colorlinks, citecolor=red]{hyperref}
\everymath{\displaystyle}
\newcommand{\R}{{\mathbb R}}
\newtheorem{theorem}{Theorem}[section]
\newtheorem{lemma}{Lemma}[section]
\newtheorem{definition}{Definition}[section]

\begin{document}

\title{\begin{flushleft} \bf Decompositions of Local mixed Morrey-type spaces\\ and Application\\[10pt] \rm{\Large Mingwei Shi and Jiang Zhou*} \end{flushleft}}
\date{\it College of Mathematics and System Sciences, Xinjiang University, Urumqi 830046, China}
\maketitle

\begin{adjustwidth}{1cm}{1cm}
	\noindent{{\bf Abstract:}{ In this paper, we obtain predual spaces of  local mixed Morrey-type spaces, characterize mixed Hardy local Morrey-type spaces. Further also, investigate nonsmooth decomposition of local mixed Morrey-type spaces. As an application, we consider the Hardy operators on local mixed Morrey-type spaces.}}
\end{adjustwidth}
\par

	\baselineskip 15pt
\section{Introduction}
\renewcommand{\thefootnote}{}
\footnotetext{*Corresponding author. The research was supported by National Natural Science Foundation of China (12061069).}
\par 
In 1938, Morrey spaces $\mathcal{M}_{p,\lambda}$ were introduced in \cite{CB1938} in relation to the research of partial differential equations. 
In 1975, D. Adams \cite{DR1975} established that Morrey spaces $\mathcal{M}_{p,\lambda}$ can describe the boundedness of the Riesz potential.
In 2004, Burenkov and Guliyev  \cite{BV2004} defined local Morrey-type spaces $L M_{p \theta, w}\left(\mathbb{R}^{n}\right)$ as follows.
\par  Let $0<p, \theta \leq \infty$ and $w$ be a non-negative measurable function on $(0, \infty)$, $f \in L M_{p \theta, w}\left(\mathbb{R}^{n}\right) $ if, $f \in L^{loc}_p \left(\mathbb{R}^{n}\right) $ and
$$
\|f\|_{L M_{p \theta, w}}:=\|f\|_{L M_{p \theta, w}\left(\mathbb{R}^{n}\right)}=\left\|w(\cdot)
\|f\|_{L_{p}(Q(0, \cdot))}\right\|_{L_{\theta}(0, \infty)} < \infty.
$$
\par
In 2004, Burenkov and Guliyev \cite{BV2004} derived a sufficient and necessary condition for the maximal operator on spaces $L M_{p \theta, w}\left(\mathbb{R}^{n}\right)$. In 2006, Burenkov et al. \cite{BV2006} received a sufficient and necessary condition for the fractional maximal operator on spaces  $L M_{p \theta, w}\left(\mathbb{R}^{n}\right)$. In 2011, Burenkov et al. \cite{BJ2011} got the boundedness of the Hardy operators on spaces  $L M_{p \theta, w}\left(\mathbb{R}^{n}\right)$. In 2017, Guliyev et al. \cite{GV2017} obtained an atomic decomposition of  spaces $L M_{p \theta, w}\left(\mathbb{R}^{n}\right)$.
\par 
In 1961, Benedek and Panzone \cite{AB1961} introduced mixed Lebesgue spaces $L^{\vec{p}}(\mathbb{R}^n) (0<\vec{p}\leq \infty)$. Recently, the mixed function spaces have new applications in partial differential equations, see \cite{cs2017,cg2017}. Various mixed function spaces are constantly proposed for reference \cite{FD1977, BA1981}. Mixed Morrey spaces $M_{\vec{q}}^{\lambda}(\mathbb{R}^n)$ were  introduced by Nogayama \cite{NT2019} in 2019. 
In 2021, Zhang and Zhou \cite{ZH2021} first defined the local mixed Morrey-type spaces $LM_{\vec{p} \theta, w}$, also investigated the boundedness of fractional intergral operators. And in 2022, Shi and Zhou \cite{SJ2022} obtained a sufficient and necessary condition about the Hardy-Littlewood maximal operator on spaces $LM_{\vec{p} \theta, w}$.
\par 
In this paper, we review local mixed Morrey-type spaces and the grand maximal operator in Section 2. In Section 3, vector valued maximal inequalities are obtained on local mixed Morrey-type spaces. In Section 4, predual spaces of local mixed Morrey-type spaces are found. In Section 5,  mixed Hardy local Morrey-type spaces are characterized. In Section 6, we attain nonsmooth decomposition on local mixed Morrey-type spaces. In Section 7, the boundedness of the Hardy operators is considered as an application. In Section 8, we compare the local mixed Morrey spaces and mixed Herz spaces, to demonstrate their norms is equivalent.

\section{Preliminaries}
	\begin{definition}(Local Mixed Morrey-type spaces) \cite{ZH2021}
	Let  $0<\vec{p},\theta\le\infty$ and $w$ be a non-negative measurable function on $(0,\infty)$. Local  mixed Morrey-type spaces were  denoted  by $LM_{\vec{p}\theta,w}(\R^n)$ respectively. For any functions $f\in L^{\rm{loc}}_1(\mathbb{R}^n)$, when $f\in LM_{\vec{p}\theta,w}(\R^n)$,  the quasi-norms
	\begin{align*}
	\|f\|_{LM_{\vec{p} \theta,w}}
	&:=\|f\|_{LM_{\vec{p} \theta,w}(\mathbb{R}^n)}:=\|w(r)\|f\|_{L_{\vec{p}}(Q(0,r))}\|_{L_{\theta}(0,\infty)}\\
	&:=\left(\int_0^\infty\left|w(r)\|f\|_{L_{\vec{p}}(Q(0,r))}\right|^{\theta}dr\right)^{\frac{1}{\theta}}<\infty.
	\end{align*}
\end{definition}

\begin{definition}(Local Mixed Morrey spaces)
	Let $\lambda \geq 0 $ and $ 0< \vec{p},\theta \leq \infty .$ We denote local mixed Morrey spaces by $LM_{\vec{p},\theta}^{\lambda}$ respectively. For any functions $f \in L^{loc}_{1}$, we say $f\in LM^{\lambda}_{\vec{p},\theta}$ when the quasi-norms
	$$
	\|f\|_{L M_{\vec{p}, \theta}^{\lambda}}=\left(\int_{0}^{\infty}\left(r^{-\lambda}\|f\|_{L_{\vec{p}}(Q(0, r))}\right)^{\theta} \frac{d r}{r}\right)^{\frac{1}{\theta}}<\infty,
	$$
when $w(r)=r^{-\lambda -\frac{1}{\theta}}$ then $LM_{\vec{p}  \theta,w}(\mathbb{R}^n) = LM_{\vec{p},\theta}^{\lambda}(\mathbb{R}^n).$
\end{definition}

\begin{lemma}\cite{ZH2021}
	Let $0<\vec{p},\theta\le\infty$ and let $w$ be a non-negative measurable function on $(0,\infty)$.
	\begin{itemize}
		\item[(\romannumeral1)] If for all $t>0$
		$$
		\|w(r)\|_{L_{\theta}(t,\infty)}=\infty,
		$$
		then $LM_{\vec{p}\theta,w}(\R^n)=\Theta,$ where $\Theta$ is the set of all functions equivalent to 0 on $\mathbb{R}^n$.
		\item[(\romannumeral2)] If for all $t>0$
		$$
		\|w(r)r^{\sum_{i=1}^{n}\frac{1}{p_i}}\|_{L_{\theta}(0,t)}=\infty,
		$$
		then $f(0)=0$ for all $f\in LM_{\vec{p}\theta,w}(\R^n)$ continuous at 0.
	\end{itemize}
\end{lemma}

\begin{definition} ($\Omega_{\theta}$ and  $\Omega_{\vec{p},\theta}$)
	Let $0<\vec{p},\theta\le\infty$. We denote by $\Omega_{\theta}$ the set of all functions $w$ which are non-negative, measurable on $(0,\infty)$, not equivalent to 0, and such that for some $t>0$
	$$\|w\|_{L_{\theta}(t,\infty)}<\infty.$$
	Moreover, we denote by $\Omega_{\vec{p},\theta}$ the set of all functions $w$ which is non-negative, measurable on $(0,\infty)$, not equivalent to 0, and such that for some $t_1,t_2>0$
	$$\|w\|_{L_{\theta}(t_1,\infty)}<\infty,~~~~\|w(r)r^{\sum_{i=1}^{n}\frac{1}{p_i}}\|_{L_{\theta}(0,t_2)}<\infty.$$
\end{definition}
\par 
In the following article, by Lemma 2.1, we always assume that either $w \in \Omega_{\theta} $ or $w \in \Omega_{\vec{p},\theta}$. $\vec{p}',\theta'$ is the conjugate index of $\vec{p},\theta.$
\begin{definition}
($\mathcal{F}_{N}, \mathcal{M} f(x)$)\cite{GV2017}
$(i)$ The topology spaces of   $\mathcal{S}\left(\mathbb{R}^{n}\right) $ is defined by the norms  $\left\{\rho_{N}\right\}_{N \in \mathbb{N}} $,
 where $
\rho_{N}(\varphi) := \sum_{|\alpha| \leq N} \sup _{x \in \mathbb{R}^{n}}(1+|x|)^{N}\left|\partial^{\alpha} \varphi(x)\right| \quad\left(\varphi \in \mathcal{S}\left(\mathbb{R}^{n}\right)\right) .
$\par
$$ \mathcal{F}_{N} := \left\{\varphi \in \mathcal{S}\left(\mathbb{R}^{n}\right): \rho_{N}(\varphi) \leq 1\right\}  ~~~\quad  N \in  \mathbb{N}\cup {\{0\}}. $$
$(ii)$The space $ \mathcal{S}^{\prime}\left(\mathbb{R}^{n}\right) $ is the topological dual of  $\mathcal{S}\left(\mathbb{R}^{n}\right) $.
\\(iii)Let  $f \in \mathcal{S}^{\prime}\left(\mathbb{R}^{n}\right) $, the grand maximal operator $ \mathcal{M} f$ of  f  is defined by
$$
\mathcal{M} f(x):=\mathcal{M}_{N} f(x) := \sup_{ t>0 ~,~ \varphi \in \mathcal{F}_{N}} \left\{\left|t^{-n} \varphi\left(t^{-1} \cdot\right) * f(x)\right|\right\}~~~~\text{for}~~ x \in \mathbb{R}^{n}. 
$$

\end{definition}
\par 
 $C_{\text{c}}^{\infty}\left(\mathbb{R}^{n}\right)$ denotes the set of all compactly supported infinitely continously differentiable functions, the set of all polynomials of degree less than or equal to  d  is denoted by $ \mathcal{P}_{d}\left(\mathbb{R}^{n}\right) $.
\begin{lemma}\cite{SE1993}
   Let  $f \in \mathcal{S}^{\prime}\left(\mathbb{R}^{n}\right)\cap L^{\mathrm{loc}}_{1}\left(\mathbb{R}^{n}\right), d \in \mathbb{N}\cup {\{0\}}$  and  $j \in \mathbb{Z} $. Then there exist an index set  $K_{j}$ , collections of cubes $ \left\{Q_{j, k}\right\}_{k \in K_{j}}$  and functions  $\left\{\eta_{j, k}\right\}_{k \in K_{j}} \subset C_{\text {c }}^{\infty}\left(\mathbb{R}^{n}\right) $, which are all indexed by $ K_{j}$  for every  j, and a decomposition
   $$
   f=g_{j}+b_{j}, \quad b_{j}=\sum_{k \in K_{j}} b_{j, k},
   $$
   such that the following properties hold.
      \begin{itemize}
      \item[(\romannumeral1)]
      $  g_{j}, b_{j}, b_{j, k} \in \mathcal{S}^{\prime}\left(\mathbb{R}^{n}\right) $.
     \item[(\romannumeral2)]
      Define  $\mathcal{O}_{j} :=\left\{y \in \mathbb{R}^{n}: \mathcal{M} f(y)>2^{j}\right\}$  and consider its Whitney decomposition. Then the cubes  $\left\{Q_{j, k}\right\}_{k \in K_{j}}$  have the bounded intersection property, and
\begin{equation}\label{88}
   \mathcal{O}_{j}=\bigcup_{k \in K_{j}} Q_{j, k}
\end{equation}
   \item[(\romannumeral3)]
    Consider the partition of unity  $\left\{\eta_{j, k}\right\}_{k \in K_{j}} $ with respect to  $\left\{Q_{j, k}\right\}_{k \in K_{j}} $. Then each function $ \eta_{j, k}$  is supported in $ Q_{j, k}$  and
   $$
   \sum_{k \in K_{j}} \eta_{j, k}=\chi_{\left\{y \in \mathbb{R}^{n}: \mathcal{M} f(y)>2^{j}\right\}}, \quad 0 \leq \eta_{j, k} \leq 1 .
   $$
      \item[(\romannumeral4)] 
      $ g_{j}$  is an  $L^{\infty}\left(\mathbb{R}^{n}\right)$  -function satisfying  $\left\|g_{j}\right\|_{L^{\infty}} \leq 2^{-j}.$
   \item[(\romannumeral5)]
    Each distribution $ b_{j, k}$  is given by  $b_{j, k}=\left(f-c_{j, k}\right) \eta_{j, k}$  with a certain polynomial $ c_{j, k} \in \mathcal{P}_{d}\left(\mathbb{R}^{n}\right) $ satisfying
   $$
   \left\langle f-c_{j, k}, \eta_{j, k} \cdot P\right\rangle=0 \text { for all } q \in \mathcal{P}_{d}\left(\mathbb{R}^{n}\right)
   $$
   and
   $$
   \mathcal{M} b_{j, k}(x) \lesssim \mathcal{M} f(x) \chi Q_{j, k}(x)+2^{j} \frac{\ell_{j, k}^{n+d+1}}{\left|x-x_{j, k}\right|^{n+d+1}} \chi_{\mathbb{R}^{n} \backslash Q_{j, k}}(x)
   $$
  $\text { for all } x \in \mathbb{R}^{n} \text {. }$
    \end{itemize}
  In the above, $ x_{j, k}$  and $ \ell_{j, k}$  denote the center and the edge-length of $ Q_{j, k} $.
\end{lemma}
\section{Vector valued maximal inequalities}
The Hardy operators $H$  and its dual operators $ H^{*} $, given by:
$$
H g(r)=\int_{0}^{r} g(t) d t \text ,~~\quad ~ H^{*} g(r)=\int_{r}^{\infty} g(t) d t .
$$
Since in the proof of Theorem 3.1 below, we mustc need the following relationship.
For  $1 \leq \vec{p} <\infty $ and a measurable function $ v:(0, \infty) \rightarrow(0, \infty) $\cite{ZH2021}, whose norm is given by
$$
\|f\|_{L_{\vec{p}, v}(0, \infty)} := \| v f\|_{L_{\vec{p}}(0, \infty)}.
$$
\begin{lemma} \cite{NT2019}
  Let $ 1<\vec{q}<\infty, 1<u \leq \infty $, for every sequence  $\left\{f_{j}\right\}_{j=1}^{\infty} \subset L^{0}\left(\mathbb{R}^{n}\right) $, where $L^{0}\left(\mathbb{R}^{n}\right) $ denotes the set of
  measureable functions on $\left(\mathbb{R}^{n}\right) $. Then
  	$$
  \left\|Mf\right\|_{L_{\vec{q}}\left(\mathbb{R}^{n}\right)} \lesssim\left\|f\right\|_{L_{\vec{q}}\left(\mathbb{R}^{n}\right)},
  $$
  and
	$$
	\left\|\left(\sum_{j=1}^{\infty}\left[M f_{j}\right]^{u}\right)^{\frac{1}{u}}\right\|_{L_{\vec{q}}\left(\mathbb{R}^{n}\right)} \lesssim\left\|\left(\sum_{j=1}^{\infty}\left|f_{j}\right|^{u}\right)^{\frac{1}{u}}\right\|_{L_{\vec{q}}\left(\mathbb{R}^{n}\right)}.
	$$
Especially $u = \infty$,
\begin{equation}\label{10}
\left\|M\left[
\mathop{sup}\limits_{j \in \mathbb{N}}|f_j|\right]\right\|_{L_{\vec{q}}\left(\mathbb{R}^{n}\right)} \lesssim\left\|\mathop{sup}\limits_{j \in \mathbb{N}}|f_j|\right\|_{L_{\vec{q}}\left(\mathbb{R}^{n}\right)}.
\end{equation}
\end{lemma}
\begin{theorem}
	Let  $1<\vec{p}<\infty, 1<\theta \leq \infty,  1<\hat{v}<\infty $, define weights $ \hat{v}_{1}, \hat{v}_{2} $ by 
	\begin{equation}\label{1}
	\hat{v}_{1}(r) := r^{-\sum_{i=1}^{n}\frac{1}{p_i}-1} w(r), \quad \hat{v}_{2}(r)=r^{-\sum_{i=1}^{n}\frac{1}{p_i}} w(r) .
	\end{equation}
 $H^{*} $ is bounded from $ L_{\theta, \hat{v}_{1}}(0, \infty)$  to $ L_{\theta, \hat{v}_{2}}(0, \infty) $. Then 
\begin{equation}\label{2}
	\|M f\|_{L M_{\vec{p} \theta, w}} \lesssim\|f\|_{L M_{\vec{p} \theta, w}}
\end{equation}
	and
\begin{equation}\label{3}
	\left\|\left(\sum_{j=1}^{\infty}\left(M f_{j}\right)^{v}\right)^{\frac{1}{v}}\right\|_{L M_{\vec{p} \theta, w}} \lesssim\left\|\left(\sum_{j=1}^{\infty}\left|f_{j}\right|^{v}\right)^{\frac{1}{v}}\right\|_{L M_{\vec{p} \theta, w}}.
\end{equation}
 In particular,
\begin{equation}\label{4}
	\left\|M\left[\sup _{j \in \mathbb{N}}\left|f_{j}\right|\right]\right\|_{L M_{\vec{p} \theta, w}} \lesssim \left\|\sup _{j \in \mathbb{N}}\left|f_{j}\right|\right\|_{L M_{\vec{p} \theta, w}.}
\end{equation}
\end{theorem}
\begin{proof}
Refer to \cite{SJ2022} to obtain \eqref{3}.
We can deduce \eqref{4} by using \eqref{3} and  \eqref{2}. By setting   $f_{1}=f, f_{2}=f_{3}=   \cdots=0 $ in \eqref{3}, we can obtain \eqref{2}. Hence concentrate on proving \eqref{3}. 
\par 
We only suppose  $\theta<\infty $; the case $ \theta=\infty$  can be dealt similarly.
$$
\left\|\chi_{Q(x, r)}\left(\sum_{j=1}^{\infty} |M f_{j}|^{v}\right)^{\frac{1}{v}}\right\|_{L^{\vec{p}}} \lesssim r^{\sum_{i=1}^{n}\frac{1}{p_i}} \int_{2 r}^{\infty} t^{-\sum_{i=1}^{n}\frac{1}{p_i}-1}\left\|\chi_{Q(x, t)}\left(\sum_{j=1}^{\infty}\left|f_{j}\right|^{v}\right)^{\frac{1}{v}}\right\|_{L^{\vec{p}}} d t .
$$
Referring to the method of \cite{ZH2021,SJ2022},  with the help of the boundedness of  $H^{*} $, we obtain the desired result.
\end{proof}	
Without the aid of the Hardy operators,
we provide a new solution  to the proof of  vector valued maximal inequalities on $LM_{\vec{p},q}^{\lambda}(\mathbb{R}^n)$.
\begin{theorem}
  Let $ 1<\vec{p}<\infty$, $ 0<q \leq \infty$,$ 1<v<\infty$  and $ 0 \leq \lambda<\sum_{i=1}^{n}\frac{1}{p_i}$ . Then 
\begin{equation}\label{5}
  \left\|\left(\sum_{j=1}^{\infty}\left(M f_{j}\right)^{v}\right)^{\frac{1}{v}}\right\|_{L M_{\vec{p}, q}^{\lambda}} \lesssim \left\|\left(\sum_{j=1}^{\infty}\left|f_{j}\right|^{v}\right)^{\frac{1}{v}}\right\|_{L M_{\vec{p}, q}^{\lambda}} .
\end{equation}
\end{theorem}

\begin{proof}
Defining   $ f_{j} :=  f_{j, 1} +f_{j, 2}$  and $f_{j, 1} := f_{j} \chi_{Q(5 r)}$. We will begin by proving the following two inequalities.

\begin{equation}\label{6}
\int _0 ^\infty \frac{1}{r^{\lambda q}} \left\| \left( \sum_{j=1}^{\infty} (Mf_{j,1})^v\right) ^{\frac{1}{v}} \right\|^q _{L_{\vec{p}}(Q(r))} \frac{dr}{r} \lesssim \left\|\left(\sum_{j=1}^{\infty}\left|f_{j}\right|^{v}\right)^{\frac{1}{v}}\right\|_{L M_{\vec{p}, q}^{\lambda}}^q 
\end{equation}

\begin{equation}\label{7}
\int _0 ^\infty \frac{1}{r^{\lambda q}} \left\| \left( \sum_{j=1}^{\infty} (Mf_{j,2})^v\right) ^{\frac{1}{v}} \right\|^q _{L_{\vec{p}}(Q(r))} \frac{dr}{r} \lesssim \left\|\left(\sum_{j=1}^{\infty}\left|f_{j}\right|^{v}\right)^{\frac{1}{v}}\right\|_{L M_{\vec{p}, q}^{\lambda}}^q 
\end{equation}
The estimate \eqref{6} follows analogously to \eqref{5}, 
\par  As for \eqref{7}, in fact
$$
\begin{aligned}
\frac{1}{r^{\lambda }} \left\| \left( \sum_{j=1}^{\infty} (Mf_{j,2})^v\right) ^{\frac{1}{v}} \right\| _{L_{\vec{p}}(Q(r))} 
& \lesssim  r^{\sum_{i=1}^{n}\frac{1}{p_i}-\lambda} \sum_{k=1}^{\infty} \frac{1}{\left|B\left(2^{k} r\right)\right|} \left\|\left(\sum_{j=1}^{\infty}\left|f_{j}\right|^{v}\right)^{\frac{1}{v}} \right\| _{L_1 (Q(2^k r) )} \\
& \lesssim  r^{\sum_{i=1}^{n}\frac{1}{p_i}-\lambda} \sum_{k=1}^{\infty} \frac{1}{\left|B\left(2^{k} r\right)\right|} \left\|\left(\sum_{j=1}^{\infty}\left|f_{j}\right|^{v}\right)^{\frac{1}{v}} \right\| _{L_{\vec{p}(Q(2^k r))}.}
 \end{aligned}
$$
Under variable $2^k r \mapsto r$, if $0<q\leq 1$, $(a+b)^q \lesssim a^q + b^q ~~a,b \in \R^+$, if $q>1$, by the H$\ddot{o}$lder inequality, then
$$
\begin{aligned}
&\int ^{\infty}_0 \left(\frac{1}{r^{\lambda }} \left\| \left( \sum_{j=1}^{\infty} (Mf_{j,2})^v\right) ^{\frac{1}{v}} \right\| _{L_{\vec{p}}(Q(r))}\right) ^q \frac{dr}{r} \\
&\quad  \lesssim  \int ^{\infty}_0 \left(r^{\sum_{i=1}^{n}\frac{1}{p_i}-\lambda} \sum_{k=1}^{\infty} \frac{1}{\left|B\left(2^{k} r\right)\right|} \left\|\left(\sum_{j=1}^{\infty}\left|f_{j}\right|^{v}\right)^{\frac{1}{v}} \right\| _{L_{\vec{p}(Q(2^k r))}}\right) ^q \frac{dr}{r}
\\&\quad \lesssim
\int ^{\infty}_0 \sum_{k=1}^{\infty} 2^{-k  \left(\sum_{i=1}^{n}\frac{1}{p_i} - \lambda\right)} \left(r^{-\lambda}  \left\|\left(\sum_{j=1}^{\infty}\left|f_{j}\right|^{v}\right)^{\frac{1}{v}} \right\| _{L_{\vec{p}(Q(2^k r))}}\right) ^q \frac{dr}{r}\\
&\quad \lesssim  \left\|\left(\sum_{j=1}^{\infty}\left|f_{j}\right|^{v}\right)^{\frac{1}{v}}\right\|_{L M_{\vec{p}, q}^{\lambda}.}^q 
\end{aligned}
$$
\end{proof}
\section{Predual spaces of local mixed Morrey-type spaces}
Let $1<\vec{p}<\infty , 1<\theta \leq \infty , w \in \Omega_{\vec{p}\theta}$.  If $supp(A)\subset Q(R)$ and $\| A \|_{L_{\vec{p}}} \leq w (R)$, which we call the function A is a $(\vec{p} , w , R)$-block. 
The local block space $LH_{\vec{p}' \theta ',w} (\mathbb{R}^n)$ is the set of all measurable functions g. There exists a decomposition $$g(x)= \sum_{j=-\infty}^{\infty} \lambda_j A_j (x).$$
Where each of A is a $(\vec{p}^{\prime} ,w,R)$-block   and $\{ \lambda _{j}  \}^{\infty}_ {j=-\infty} \in  l^{\theta '}$ and the convergence of almost all $x\in (\mathbb{R}^n)$, the norm of g is given by:
$$
\|g\| _{LH_{\vec{p}^{\prime} \theta ',w}} := inf \left(\sum_{j=-\infty}^{\infty} |\lambda_j|^{\theta '}\right) ^{\frac{1}{\theta '}},
$$
where $\{ \lambda _{j}  \}^{\infty}_ {j=-\infty}$ covers all the above admissible expressions.
\begin{theorem}
	Let $ 1<\vec{p}<\infty$, $1<\theta \leq \infty $ and $ w \in \Omega_{\theta} $. Assume that  w  satisfies the doubling condition; $ C^{-1} w(r) \leq w(2 r) \leq C w(r)$  for all $ r>0 $,  then $ L M_{\vec{p} \theta, w}\left(\mathbb{R}^{n}\right) $ is the dual of $ L H_{\vec{p}^{\prime} \theta^{\prime}, w} \left(\mathbb{R}^{n}\right) $ in the following sense:
	      \begin{itemize}
		\item[(\romannumeral1)]
	 Let $ f \in L M_{\vec{p} \theta, w}\left(\mathbb{R}^{n}\right) $, for any  $g \in LH_{\vec{p}^{\prime} \theta^{\prime},w }\left(\mathbb{R}^{n}\right) $, then $f g \in   L^{1}\left(\mathbb{R}^{n}\right) $, the mapping
	$$
	g \in L H_{\vec{p} ^{\prime} \theta^{\prime}, w }\left(\mathbb{R}^{n}\right) \mapsto \int_{\mathbb{R}^{n}} f(x) g(x) d x \in \mathbb{C}
	$$
may	define a continuous linear functional $ L_{f} $ on $ L H_{\vec{p} ^{\prime} \theta^{\prime}, w }\left(\mathbb{R}^{n}\right) $.
		\item[(\romannumeral2)]
	 Conversely, any continuous linear functional  L  on $ L H_{\vec{p} ^{\prime} \theta^{\prime}, w }\left(\mathbb{R}^{n}\right) $ can be realized as $ L=L_{f}\left(\mathbb{R}^{n}\right) $ with a certain  $f \in L M_{\vec{p} \theta, w}\left(\mathbb{R}^{n}\right) $. 
	\end{itemize}
Futhermore, for all $ f \in L M_{\vec{p} \theta, w}\left(\mathbb{R}^{n}\right) $ the operator norm of $L_f$ is equivalent to $\|f\|_{L M_{\vec{p} \theta, w}}$, scilicet there exists a constant $ C>0$  such that
$$
C^{-1}\|f\|_{L M_{\vec{p} \theta, w}} \leq\left\|L_{f}\right\|_{L H_{\vec{p} ^{\prime} \theta^{\prime}, w } \rightarrow \mathbb{C}} \leq C\|f\|_{L M_{\vec{p} \theta, w}.}
$$

\end{theorem} 
\begin{proof}
(1)
Let  g  be such that
$$g:=\sum_{j=-\infty}^{\infty} \lambda_{j} A_{j},$$
where each $ A_{j}$  is a  $\left(\vec{p}^{\prime}, w, 2^{j}\right) $-block and $ \left\{\lambda_{j}\right\}_{j=1}^{\infty} \in l^{\theta^{\prime}}$  satisfies
$$
\left(\sum_{j=-\infty}^{\infty}\left|\lambda_{j}\right|^{\theta^{\prime}}\right)^{\frac{1}{\theta^{\prime}}} \leq 2\|g\|_{L H_{\vec{p}^{\prime} \theta^{\prime}, w }} .$$\\
Then
$$
\begin{aligned}
	\|f   g\|_{L_{1}(\mathbb{R}^n)} & \leq \sum_{j=-\infty}^{\infty}\left|\lambda_{j}\right| \int_{B\left(2^{j}\right)}\left|f(x) A_{j}(x)\right| d x \\
	 &\leq \sum_{j=-\infty}^{\infty}\left|\lambda_{j}\right|
	\|f\|_{L_{\vec{p}}\left(B_j\right)} \|A_j (x)\|_{L^{\vec{p}^\prime}(B_j)}\\
    &\leq \sum_{j=-\infty}^{\infty}\left|\lambda_{j}\right| w(2^j)
	\|f\|_{L_{\vec{p}}\left(B_j\right)}
	\leq \left(\sum_{j=-\infty}^{\infty}\left|\lambda_{j}\right| ^{\theta'}\right) ^{\frac{1}{\theta ^\prime}}
	\left(\sum_{j=-\infty}^{\infty}w(2^j)^{\theta} 
	\|f\|_{L_{\vec{p}}\left(B_j\right)} ^{\theta}\right) ^{\frac{1}{\theta}}\\
	&\leq \left(\sum_{j=-\infty}^{\infty}\left|\lambda_{j}\right| ^{\theta'}\right) ^{\frac{1}{\theta ^\prime}}
	\left(\sum_{j=-\infty}^{\infty}w(2^j)^{\theta} 
	\|f\|_{L_{\vec{p}}\left(B_j\right)} ^{\theta}\right) ^{\frac{1}{\theta}}\\&
	\lesssim  \left(\sum_{j=-\infty}^{\infty}\left|\lambda_{j}\right| ^{\theta'}\right) ^{\frac{1}{\theta ^\prime}}
	\left(\int ^\infty _0 \left( w(r) 
	\|f\|_{L_{\vec{p}}\left(B(r)\right)} \right) ^{\theta}\right) ^{\frac{1}{\theta}}\\
	& \lesssim \|f\|_{LM_{\vec{p}\theta,w}} \|g\|_{LH_{\vec{p}^\prime \theta^\prime  ,w}}.
\end{aligned}
$$
(2) 
Let L be a bounded linear functional on $ LH_{\vec{p}^\prime \theta^\prime  ,w}(\R^n)$, since the mapping
$$
g \in L^{\vec{p}^{\prime}}\left(\mathbb{R}^{n}\right) \mapsto L\left(g \chi_{B\left(2^{j}\right)}\right) \in \mathbb{C}
$$
is a bounded linear functional, we see that L is realized by an $L^{loc}_{\vec{p}}
(\R^n)$-function
$f$ satisfy
\begin{equation}\label{10}
L(g \chi_{B(2^j)}) = \int _{B(2^j)} g(x)f(x)dx
\end{equation}
for all $ g \in L_{\vec{p}^{\prime}}\left(\mathbb{R}^{n}\right) $ and $ j \in \mathbb{Z} $. We have to check  $f \in L M_{\vec{p} \theta, w }\left(\mathbb{R}^{n}\right) $, or equivalently,
$$\left(
\sum_{j=-\infty}^{\infty}w(2^j)^{\theta}
\|f\|_{L_{\vec{p}}\left(B_j\right)} ^{\theta}\right) ^{\frac{1}{\theta}} < \infty 
$$
To this end, choose a nonnegative $ \ell^{\theta^{\prime}}(\mathbb{Z})$ -sequence $ \left\{\rho_{j}\right\}_{j=-\infty}^{\infty} $ arbitrarily so that $ \rho_{j}=0 $ with  $|j| \gg 1$  and estimate
$$
\sum_{j=-\infty}^{\infty} w(2^j) \rho_{j}
\|f\|_{L_{\vec{p}}\left(B_j\right)}
$$
Let us set
$$
g_{j}(x) :=\left\{\begin{array}{ll}
	 (sgnf) w\left(2^{j}\right) |f(x)|^{p_1-1} \chi_{B\left(2^{j}\right)}(x) (\| f\|_{p_1})^{p_2-p_1} (\|f\|_{p_1,p_2})^{p_3-p_2}\cdot \cdot \cdot(\|f\|_{p_1\cdot \cdot \cdot p_n})^{1-p_n} , & \text { if }\|f\|_{L^{\vec{p}}\left(B\left(2^{j}\right)\right)}>0 \\
	0, & \text { otherwise. }
\end{array}\right.
$$
Then each $ g_{j}$  is a $ \left(\vec{p}^{\prime}, w, R\right) $-block
$$
g := \sum_{j=-\infty}^{\infty} \rho_{j} g_{j}
$$
belongs to  $L H_{\vec{p}^{\prime} \theta^{\prime}, w }\left(\mathbb{R}^{n}\right) $ and satisfies
$$
\int_{\mathbb{R}^{n}}|f(x) g(x)| d x=\sum_{j=-\infty}^{\infty} w\left(2^{j}\right) \rho_{j}\|f\|_{L_{\vec{p}}\left(B_j\right)}
$$
Therefore, by letting $ h(x) := \overline{\operatorname{sgn}(f(x))} g(x) $ for $ x \in \mathbb{R}^{n} $, since  $\operatorname{supp}(h) \subset B\left(2^{j}\right) $ for some large  j,
$
\int_{\mathbb{R}^{n}}|f(x) g(x)| d x=L(h)
$
thanks to \eqref{10} and the fact that $ \rho_{j}=0  $ if $ |j| \gg 1$ . Thus
$$
\begin{aligned}
	\sum_{j=-\infty}^{\infty} w\left(2^{j}\right) \rho_{j}\|f\|_{L_{\vec{p}}\left(B_j\right)}&=\int_{\mathbb{R}^{n}}|f(x) g(x)| d x\\
	&=L(h) \leq\|L\|_{L H_{\vec{p}^{\prime} \theta^{\prime}, w } \rightarrow \mathbb{C}}\| h\|_{LH_{\vec{p}^\prime \theta^\prime, w}}
	\\&\leq\|L\|_{L H_{\vec{p}^{\prime} \theta^{\prime}, w } \rightarrow \mathbb{C}}\left(\sum_{j=-\infty}^{\infty}\left|\rho_{j}\right|^{\theta^{\prime}}\right)^{\frac{1}{\theta^{\prime}}} .
\end{aligned}
$$
\end{proof}

\section{Charcterization of mixed Hardy local  Morrey-type spaces in terms of the grand maxiaml operators and the heat kernel}
Let $ 1<\vec{p}<\infty, 1<\theta \leq \infty$  and $ w \in \Omega_{\vec{p} \theta} $. We characterize the space  $L M_{\vec{p} \theta, w}\left(\mathbb{R}^{n}\right) $ in terms of the heat kernel. Let $ t>0 $ and  $f \in \mathcal{S}^{\prime}\left(\mathbb{R}^{n}\right) $ \cite{DR1975}
$$
e^{t \Delta} f(x) :=\left\langle f, \frac{1}{\sqrt{(4 \pi t)^{n}}} \exp \left(-\frac{|x-\cdot|^{2}}{4 t}\right)\right\rangle \quad x \in \mathbb{R}^{n} .
$$
 \begin{definition}
	(the mixed Hardy local Morrey-type space) The mixed Hardy local Morrey-type space $H L M_{\vec{p} \theta, w}\left(\mathbb{R}^{n}\right)$ collects all $f \in \mathcal{S}^{\prime}\left(\mathbb{R}^{n}\right)$ such that $\sup _{t>0}\left|e^{t \Delta} f\right| \in L M_{\vec{p} \theta, w}\left(\mathbb{R}^{n}\right)$.
	$$
	\|f\|_{H L M_{\vec{p} \theta, w}} :=\left\|\sup _{t>0}\left|e^{t \Delta} f\right|\right\|_{L M_{\vec{p} \theta, w}} < \infty.
	$$
\end{definition}
Let us show that $L M_{\vec{p} \theta, w}\left(\mathbb{R}^{n}\right)$ and $H L M_{\vec{p} \theta, w}\left(\mathbb{R}^{n}\right)$ are isomorphic.
\begin{theorem}
	Let $1<\vec{p}<\infty, 1<\theta \leq \infty$, $w \in \Omega_{\theta}$, define $\hat{v}_{1}, \hat{v}_{2}$ by \eqref{18} and  $H^{*}$ is bounded from $L_{\theta, \hat{v}_{1}}(0, \infty)$ to $L_{\theta, \hat{v}_{2}}(0, \infty)$.
		      \begin{itemize}
		\item[(\romannumeral1)]
 If $f \in L M_{\vec{p} \theta, w}\left(\mathbb{R}^{n}\right)$, then $f \in H L M_{\vec{p} \theta, w}\left(\mathbb{R}^{n}\right)$.
 	\item[(\romannumeral2)]
 If $f \in H L M_{\vec{p} \theta, w}\left(\mathbb{R}^{n}\right)$, then $f$ is represented by a measurable function $g\in L M_{\vec{p} \theta,w }\left(\mathbb{R}^{n}\right)$.
\end{itemize}
	If $f \in L M_{\vec{p} \theta, w}\left(\mathbb{R}^{n}\right)$, then
\begin{equation}\label{11}
	\|f\|_{L M_{\vec{p} \theta, w}} \leq\|f\|_{H L M_{\vec{p} \theta, w}} \leq C\|f\|_{L M_{\vec{p} \theta, w}}
\end{equation}
\end{theorem}
\begin{proof}
 (1) Easily verify that $L M_{\vec{p} \theta, w}\left(\mathbb{R}^{n}\right) \hookrightarrow \mathcal{S}^{\prime}\left(\mathbb{R}^{n}\right)$ in the sense of continuous embedding by  Lemma 6.2. Also\cite{GV2017}
$
\sup _{t>0}\left|e^{t \Delta} f\right| \leq M f .
$
\par 
From \eqref{3}, the $L M_{\vec{p} \theta, w}\left(\mathbb{R}^{n}\right)$-boundedness of the Hardy-Littlewood maximal operator, we see that $f \in H L M_{\vec{p} \theta, w}\left(\mathbb{R}^{n}\right)$ and that the right inequality in \eqref{11} follows.\par
(2) Due to Theorem 4.1, the dual of $L H_{\vec{p}^{\prime} \theta^{\prime},, w }\left(\mathbb{R}^{n}\right)$ is isomorphic to $L M_{\vec{p} \theta, w}\left(\mathbb{R}^{n}\right)$. Let $L: h \in L M_{\vec{p} \theta, w}\left(\mathbb{R}^{n}\right) \mapsto L_{h} \in\left(L H_{\vec{p}^{\prime} \theta^{\prime}, \bar{w} }\left(\mathbb{R}^{n}\right)\right)^{*}$ be an isomorphism in Theorem 4.1. By the Banach-Alaoglu theorem, there exists a positive decreasing sequence $\left\{t_{j}\right\}_{j=1}^{\infty} \subset(0,1)$ such that $L_{e^{t} \Delta_{f}}$ is convergent to $G=L_{g} \in$ $\left(L H_{\vec{p}^{\prime} \theta^{\prime}, \bar{w} }\left(\mathbb{R}^{n}\right)\right)^{*}$ for some $g \in L M_{\vec{p} \theta, w}\left(\mathbb{R}^{n}\right)$ in the weak-* sense. Observe that
$$
\begin{aligned}
\|g\|_{L M_{\vec{p} \theta, w(.)}} &\sim\left\|L_{g}\right\|_{\left(L H_{\vec{p}^{\prime} \theta^{\prime}, \tilde{w}}\right)^{*}}\\
&\leq \liminf _{j \rightarrow \infty}\left\|L_{e^{t} \Delta^{\Delta}}\right\|_{\left(L H_{\vec{p}^{\prime} \theta^{\prime}, \bar{w} }\right)^{*}}\\
&\sim \liminf _{j \rightarrow \infty}\left\|e^{t_{j} \Delta} f\right\|_{L M_{\vec{p} \theta, w}} \leq\|f\|_{H L M_{\vec{p} \theta, w}} .
\end{aligned}
$$
Meanwhile, since $f \in \mathcal{S}^{\prime}\left(\mathbb{R}^{n}\right), e^{t_{j} \Delta} f$ is convergent to $f \in \mathcal{S}^{\prime}\left(\mathbb{R}^{n}\right)$. Thus, we conclude $\mathcal{S}^{\prime}\left(\mathbb{R}^{n}\right) \ni f=g \in L M_{\vec{p} \theta, w}\left(\mathbb{R}^{n}\right)$.
The left inequality in \eqref{11} follows since the spaces $L M_{\vec{p} \theta, w}\left(\mathbb{R}^{n}\right)$ is isomorphic to the dual of $L H_{\vec{p}^{\prime} \theta^{\prime}, \bar{w} }\left(\mathbb{R}^{n}\right)$. Thus, from Lebesgue's differentiation theorem,
$$
\|f\|_{L M_{\vec{p} \theta, w}} \leq\left\|\sup _{t>0}\left|e^{t \Delta} f\right|\right\|_{L M_{\vec{p} \theta, w}}=\|f\|_{H L M_{\vec{p} \theta, w}}.
$$
\end{proof}
In terms of the grand maximal opetator in Definition 2.4,can rephrase Theorem $5.1$ as follows:

\begin{theorem}
 Let $1<\vec{p}<\infty, 1<\theta \leq \infty$, $w \in \Omega_{\theta}$, define $\hat{v}_{1}, \hat{v}_{2}$ by \eqref{18} and $H^{*}$ is bounded from $L_{\theta, \hat{v}_{1}}(0, \infty)$ to $L_{\theta, \hat{v}_{2}}(0, \infty)$.
		     \begin{itemize}
	\item[(\romannumeral1)]
 If $f \in L M_{\vec{p} \theta, w}\left(\mathbb{R}^{n}\right)$, then $\mathcal{M} f \in L M_{\vec{p} \theta, w}\left(\mathbb{R}^{n}\right)$.
	\item[(\romannumeral2)]
Let $f \in \mathcal{S}^{\prime}\left(\mathbb{R}^{n}\right)$, if $\mathcal{M} f \in L M_{\vec{p} \theta, w}\left(\mathbb{R}^{n}\right)$, then $f$ is represented by a measurable function $g\in L M_{\vec{p} \theta, w}\left(\mathbb{R}^{n}\right)$.
\end{itemize}
If $f \in L M_{\vec{p} \theta, w}\left(\mathbb{R}^{n}\right)$, then $C^{-1}\|f\|_{L M_{\vec{p} \theta, w}} \leq\|\mathcal{M} f\|_{L M_{\vec{p} \theta, w}} \leq C\|f\|_{L M_{\vec{p} \theta, w \cdot}} .$
\end{theorem}
\begin{proof}
 The implication ($i$) $\Longrightarrow(ii)$ immediately follows from the pointwise inequality $\mathcal{M} f(x) \lesssim  M f(x)$. The converse implication ($ii$) $\Longrightarrow$ ($i$) follows from the pointwise estimatee $\left|e^{t \Delta} f(x)\right| \lesssim \mathcal{M} f(x)$. Indeed, from this pointwise estimate, we conclude $
 \mathop{sup}\limits_{t>0}\left|e^{t \Delta} f\right| \in L M_{\vec{p} \theta, w}\left(\mathbb{R}^{n}\right)$. Thus, we are in the position of applying Theorem $5.1$ to receive $f \in L M_{\vec{p} \theta, w}\left(\mathbb{R}^{n}\right)$.
\end{proof}

\section{Atomic decomposition of local mixed morrey-type spaces}
\begin{theorem}
 Let $ 1<\vec{p}<\infty, 1<\theta \leq \infty $, $ w \in \Omega_{\vec{p} \theta} $ where  $ C^{-1} w(r) \leq w(2 r) \leq C w(r) $ for all  $r>0 $. 
\begin{equation} \label{18}
\hat{v}_{1}(r) := r^{-\sum_{i=1}^{n}\frac{1}{p_i}-1} w(r), \quad \hat{v}_{2}(r)=r^{-\sum_{i=1}^{n}\frac{1}{p_i}} w(r) .
\end{equation}
 $H^{*}$  is bounded from  $L_{\theta, \hat{v}_{1}}(0, \infty) $ to  $L_{\theta, \hat{v}_{2}}(0, \infty) $,   $\vec{s}$ satisfies
\begin{equation}\label{12}
\int_{r}^{\infty} \frac{t^{\frac{1}{n}\sum_{i=1}^{\infty}\frac{1}{s_i} -\frac{1}{n}\sum_{i=1}^{n}\frac{1}{p_i}-1}}{w(r t)} d t \lesssim \frac{r^{\frac{1}{n}\sum_{i=1}^{n}\frac{1}{p_i}-\frac{1}{n}\sum_{i=1}^{n}\frac{1}{s_i}}}{w(r)} ~~\text{for all} ~~ r>0 .
\end{equation}
 And assume that  $ \left\{Q_{j}\right\}_{j=1}^{\infty} \subset \mathcal{Q}\left(\mathbb{R}^{n}\right),\left\{a_{j}\right\}_{j=1}^{\infty} \subset L^{\vec{s}}\left(\mathbb{R}^{n}\right) ,  \left\{\lambda_{j}\right\}_{j=1}^{\infty} \subset[0, \infty)$  satisfying
$$
\left\|a_{j}\right\|_{L^{\vec{s}}} \leq\left\|\chi_{Q_j}\right\|_{L^{\vec{s}}}=\left|Q_{j}\right|^{\frac{1}{n}\sum_{i=1}^{n}\frac{1}{s_i}}, \quad \operatorname{supp}\left(a_{j}\right) \subset Q_{j}, \quad \sum_{j=1}^{\infty} \| \lambda_{j} \chi_{Q_{j}} \|_{L M_{\vec{p} \theta, w}}<\infty.
$$
Then the series  $f := \sum_{j=1}^{\infty} \lambda_{j} a_{j}$  converges in $ L_{\mathrm{loc}}^{1}\left(\mathbb{R}^{n}\right) $ and in the Schwartz space $ \mathcal{S}^{\prime}\left(\mathbb{R}^{n}\right) $ of tempered distributions and satisfies the estimate
\begin{equation}\label{13}
\|f\|_{L M_{\vec{p} \theta, w}} \lesssim \left\|\sum_{j=1}^{\infty} \lambda_{j} \chi Q_{j}\right\|_{L M_{\vec{p} \theta, w}.}
\end{equation}

\end{theorem}
\begin{lemma}
Let $1<\vec{p}<\infty, 1<\theta \leq \infty$, $w \in \Omega_{\theta}$, each $A_{j}$ be a $\left(\vec{p} ', w, 2^{j}\right)$-block and $\left\{\rho_{j}\right\}_{j=-\infty}^{\infty} \in \ell^{\theta^{\prime}}$. Suppose $s_i \in(p_i, \infty)$ and $\vec{s}$ satisfies \eqref{12} for all $r>0$. Then
$$
h := \sum_{j=-\infty}^{\infty} \rho_{j}\left(M\left[\left|A_{j}\right|^{\vec{s}^{\prime}}\right]\right)^{\frac{1}{\vec{s}}} \in L H_{\vec{p}^{\prime} \theta^{\prime}, w }\left(\mathbb{R}^{n}\right)
,~~~~
\|h\|_{L H_{\vec{p}^{\prime} \theta^{\prime}, w }} \leq C\left(\sum_{j=-\infty}^{\infty}\left|\rho_{j}\right|^{\theta^{\prime}}\right)^{1 / \theta^{\prime}}.
$$
\end{lemma}
\begin{proof}
By the $L^{\vec{s}^{\prime}}\left(\mathbb{R}^{n}\right)$-boundedness of the Hardy-Littlewood maximal operator and $\theta^{\prime}<\infty$, 
$$
\sum_{j=-\infty}^{\infty} \rho_{j} \chi_{B(2 j+1)}\left(M\left[\left|A_{j}\right|^{\vec{s}^{\prime}}\right]\right)^{\frac{1}{\vec{s}}} \in L H_{\vec{p}^{\prime} \theta^{\prime}, w }\left(\mathbb{R}^{n}\right)
$$
and
$$
\left\|\sum_{j=-\infty}^{\infty} \rho_{j} \chi_{B\left(2^{j+1}\right)}\left(M\left[\left|A_{j}\right|^{\vec{s}^{\prime}}\right]\right)^{\frac{1}{\vec{s}}} \right\|_{L H_{\vec{p}^{\prime} \theta^{\prime}, \bar{w} }} \lesssim \left(\sum_{j=-\infty}^{\infty}\left|\rho_{j}\right|^{\theta^{\prime}}\right)^{1 / \theta^{\prime}}.
$$
Meanwhile, combining $\left\|A_{j}\right\|_{L^{\vec{s}^{\prime}}} \leq\left|B\left(2^{j}\right)\right|^{\frac{1}{n}\sum_{i=1}^{n}\frac{1}{p_i}-\frac{1}{n}\sum_{i=1}^{n}\frac{1}{s_i}}\left\|A_{j}\right\|_{L_{\vec{p}^{\prime}}} \lesssim 2^{\sum_{i=1}^{n}\frac{1}{p_i}-\sum_{i=1}^{n}\frac{1}{s_i}} w\left(2^{j}\right)$ and \eqref{12}, therefore,
$$
\begin{aligned}
&\sum_{j=-\infty}^{\infty} \rho_{j} \chi_{B\left(2^{j+1}\right)^{c}}\left(M\left[\left|A_{j}\right|^{\vec{s}^{\prime}}\right]\right)^{\frac{1}{\vec{s}}}\\
&=\sum_{k=0}^{\infty} \sum_{j=-\infty}^{\infty} \rho_{j} \chi_{B\left(2^{j+k+2}\right) \backslash B\left(2^{j+k+1}\right)}\left(M\left[\left|A_{j}\right|^{\vec{s}^{\prime}}\right]\right)^{\frac{1}{\vec{s}}} \\
& \leq C \sum_{k=0}^{\infty} \sum_{j=-\infty}^{\infty} \rho_{j} \frac{1}{2^{(j+k)\sum_{i=1}^{n}\frac{1}{s_i^{\prime}}}}\left\| A_j  \right\|_{L_{\vec{s}}B((2^j))}
\chi_{B\left(2^{j+k+2}\right) \backslash B\left(2^{j+k+1}\right)} \\
&\leq C \sum_{k=0}^{\infty} \sum_{j=-\infty}^{\infty} \rho_{j} \frac{2^{j \sum_{i=1}^{n}\frac{1}{p_i}-j \sum_{i=1}^{n}\frac{1}{s_i}}}{2^{(j+k)\sum_{i=1}^{n}\frac{1}{s_i^{\prime}}}}\left\| A_j  \right\|_{L_{\vec{p}'}B((2^j))} \chi_{B\left(2^{j+k+2}\right) \backslash B\left(2^{j+k+1}\right)^{.}} \\
&\leq C \sum_{k=0}^{\infty} \sum_{j=-\infty}^{\infty} \rho_{j} \frac{2^{j \sum_{i=1}^{n}\frac{1}{p_i}-j \sum_{i=1}^{n}\frac{1}{s_i}} w\left(2^{j}\right)}{2^{(j+k)\sum_{i=1}^{n}\frac{1}{s_i^{\prime}}}} \chi_{B\left(2^{j+k+2}\right) \backslash B\left(2^{j+k+1}\right)} 
\lesssim \sum_{j=-\infty}^{\infty} \rho_{j}.
\end{aligned}
$$

\end{proof}
Next, to prove Theorem 6.1.
\begin{proof}
To prove \eqref{13}, we resort to the duality obtained in Theorem 4.1
$$
\|f\|_{L M_{\vec{p} \theta, w}} = \sup \left\{\left|\int_{\mathbb{R}^{n}} f(x) g(x) d x\right|:\|g\|_{L H_{\vec{p}^{\prime} \theta^{\prime}, w }}=1\right\} .
$$
We can assume that $\left\{\lambda_{j}\right\}_{j=1}^{\infty}$ is finitely supported thanks to the monotone convergence theorem. Let us assume in addition that the $a_{j}$ are non-negative without loss of generality. 
$$
g := \sum_{k=-\infty}^{\infty} \rho_{k} A_{k}, \quad G := \sum_{k=-\infty}^{\infty}\left|\rho_{k}\right| M\left[\left|A_{k}\right|^{\vec{s}^{\prime}}\right]^{\frac{1}{\vec{s}^{\prime}}},
$$
where each $A_{k}$ is a $\left(\vec{p}^{\prime}, w, 2^{k}\right)$-block, Lemma 6.1 and
$$
\sum_{k=-\infty}^{\infty}\left|\rho_{k}\right|^{\theta^{\prime}} \leq 1.
$$
Then 
$$
\begin{aligned}
\left|\int_{\mathbb{R}^{n}} f(x) g(x) d x\right| & \leq \sum_{(j, k) \in \mathbb{N} \times \mathbb{Z}} \lambda_{j}\left|\rho_{k}\right| \int_{B\left(2^{k}\right) \cap Q_{j}} a_{j}(x)\left|A_{k}(x)\right| d x \\
& \leq \sum_{(j, k) \in \mathbb{N} \times \mathbb{Z}} \lambda_{j}\left|\rho_{k}\right| \cdot\left\|a_{j}\right\|_{L^{\vec{s}}\left(Q_{j}\right)}\left\|A_{k}\right\|_{L^{\vec{s}^{\prime}}\left(Q_{j}\right)} \\
& \lesssim \sum_{(j, k) \in \mathbb{N} \times \mathbb{Z}} \lambda_{j}\left|\rho_{k}\right| \cdot \int_{Q_{j}} M\left[A_{k} {\vec{s}^{\prime}}\right](x)^{\frac{1}{\vec{s}^{\prime}}} d x \\
&< \infty
\end{aligned}
$$ 
\end{proof}
\begin{lemma}
Let $ 1<\vec{p}<\infty, 0<\theta \leq \infty, w \in \Omega_{\theta} $, define $ \hat{v}_{1}, \hat{v}_{2} $ by \eqref{18} and $  H^{*} $ is bounded from $ L_{\theta, \hat{v}_{1}}(0, \infty)$  to $ L_{\theta, \hat{v}_{2}}(0, \infty) $. Then  $LM_{ \vec{p} \theta, w}\left(\mathbb{R}^{n}\right) \hookrightarrow   \mathcal{S}^{\prime}\left(\mathbb{R}^{n}\right)$  in the sense of continuous embedding.
\end{lemma}
\begin{proof}
Denote by $\mathcal{B}_{x}$ the set of all open balls in $\mathbb{R}^{n}$ which contain $x$. 
The Hardy-Littlewood maximal operator $M$ is bounded from $L M_{\vec{p} \theta, w }\left(\mathbb{R}^{n}\right)$ from  \cite{SJ2022},  $\alpha :=\left\|\chi_{B(1)}\right\|_{L M_{\vec{p} \theta, w}}$. Therefore,
$$
\frac{\alpha}{|B(R)|} \int_{B(R)}|f(y)| d y \leq\left\|\chi_{B(1)} M f\right\|_{L M_{\vec{p} \theta, w}} \leq\|M f\|_{L M_{\vec{p} \theta, w }} \leq C\|f\|_{L M_{\vec{p} \theta, w}},
$$
 For all $\kappa \in \mathcal{S} (\R^n)$ and $f \in L M_{\vec{p} \theta, w} $. Then
$$
\begin{aligned}
\int_{\mathbb{R}^{n}}|\kappa(x) f(x)| d x &=\int_{B(1)}|\kappa(x) f(x)| d x+\sum_{j=1}^{\infty} \int_{B(j+1) \backslash B(j)}|\kappa(x) f(x)| d x \\
& \leq\|\kappa\|_{L^{\infty}(B(1))}\|f\|_{L^{1}(B(1))}+\sum_{j=1}^{\infty} \int_{B(j+1) \backslash B(j)} \frac{|x|^{2 n+1}}{j^{2 n+1}}|\kappa(x) f(x)| d x \\
& \lesssim \|f\|_{L M_{\vec{p} \theta, w}}\left(\sup _{x \in \mathbb{R}^{n}}(1+|x|)^{2 n+1}|\kappa(x)|\right).
\end{aligned}
$$
\end{proof}
With the aid of Lemma 6.2, we extend into Theorem 6.2.\\
\begin{theorem}
Satisfying the conditions of theorem 6.1 but where $ \left\{a_{j}\right\}_{j=1}^{\infty} \subset L^{\infty}\left(\mathbb{R}^{n}\right) $ such that $ f := \sum_{j=1}^{\infty} \lambda_{j} a_{j}$  converges in $ \mathcal{S}^{\prime}\left(\mathbb{R}^{n}\right) \cap L^{\mathrm{loc}}_{1}\left(\mathbb{R}^{n}\right) $, that
\begin{equation}\label{78}
\left|a_{j}\right| \leq \chi_{Q_{j}}, \quad \int_{\mathbb{R}^{n}} x^{\alpha} a_{j}(x) d x=0,
\end{equation}
for all multi-indices $ \alpha=\left(\alpha_{1}, \alpha_{2}, \ldots, \alpha_{n}\right) $ with  $|\alpha| := \alpha_{1}+\alpha_{2}+\cdots+\alpha_{n} < \infty $ and, that for all  $v>0$ 
\begin{equation}\label{79}
\left\|\left(\sum_{j=1}^{\infty}\left(\lambda_{j} \chi_{Q_{j}}\right)^{v}\right)^{1 / v}\right\|_{L M_{\vec{p} \theta, w}} \leq C_{v}\|f\|_{L M_{\vec{p} \theta, w}} .
\end{equation}
Here the constant  $C_{v}>0$  is independent of  f .
\end{theorem}

\begin{lemma}\cite{BS2014}
 Let $\varphi \in \mathcal{S}\left(\mathbb{R}^{n}\right)$. With the same notation as Lemma $2.2$, then
\begin{equation}\label{30}
\left|\left\langle b_{j}, \varphi\right\rangle\right| \leq C_{\varphi}\left\{\sum_{l=0}^{\infty}\left(\frac{1}{2^{l n}}\left\|\mathcal{M} f \cdot \chi \mathcal{O}_{j}\right\|_{L^{1}\left(B\left(2^{l}\right)\right)}\right)^{\mu}\right\}^{1 / \mu},
\end{equation}
where $\mu := \frac{n+d+1}{n}$ and the constant $C_{\varphi}$ in $\eqref{30}$ depends on $\varphi$ but not on $j$ or $k$.
\end{lemma}

\begin{lemma}
Let $1<\vec{p}<\infty, 1<\theta \leq \infty$, $w \in \Omega_{\theta}$, define $\hat{v}_{1}, \hat{v}_{2}$ by \eqref{18}, $H$ is bounded from $L_{\theta, \hat{\nu}_{1}}(0, \infty)$ to $L_{\theta, \hat{v}_{2}}(0, \infty)$, $f \in L M_{\vec{p} \theta, w}\left(\mathbb{R}^{n}\right)$ and  $\vec{s}$ satisfies \eqref{12} for all $r>0$. Then in the notation of Lemma 2.2, in the topology of $\mathcal{S}^{\prime}\left(\mathbb{R}^{n}\right)$,  $g_{j} \rightarrow 0$ as $j \rightarrow-\infty$ and $b_{j} \rightarrow 0$ as $j \rightarrow \infty$. In particular
$$
f=\sum_{j=-\infty}^{\infty}\left(g_{j+1}-g_{j}\right).
$$

\end{lemma}
\begin{proof}
Observe that
$$
\begin{aligned}
\frac{1}{2^{l n}}\left\|\mathcal{M} f \cdot \chi \mathcal{O}_{j}\right\|_{L^{1}\left(B\left(2^{l}\right)\right)} &\lesssim \frac{1}{2^{l n}}\|\mathcal{M} f\|_{L^{1}\left(B\left(2^{l}\right)\right)}\\
&\lesssim \frac{1}{2^{l
		\sum_{i=1}^{n}\frac{1}{p_i}}}\|\mathcal{M} f\|_{L^{\vec{p}}\left(B\left(2^{l}\right)\right)}\\
&\lesssim \frac{1}{2^{l
		\sum_{i=1}^{n}\frac{1}{p_i}} w\left(2^{l}\right)}\|f\|_{H L M_{\vec{p} \theta, w}}\\
&\lesssim \frac{1}{2^{l
		\sum_{i=1}^{n}\frac{1}{p_i}} w\left(2^{l}\right)}\|f\|_{L M_{\vec{p} \theta, w }}.
\end{aligned}
$$
Note that \eqref{12} and $\mu := \frac{n+d+1}{n}$.
$$
\sum_{l=1}^{\infty}\left(\frac{1}{2^{\l
		\sum_{i=1}^{n}\frac{1}{p_i}} w\left(2^{l}\right)}\right)^{\mu}<\infty
$$
Consequently, we may use the Lebesgue convergence theorem to conclude that $b_{j} \rightarrow 0$ as $j \rightarrow \infty$. Hence, it follows that $f=\lim _{j \rightarrow \infty} g_{j}$ in $\mathcal{S}^{\prime}\left(\mathbb{R}^{n}\right)$. Consequently, it follows from Lemma 2.2 (4) that $f=\lim _{j \rightarrow \infty} g_{j}=\lim _{j, k \rightarrow \infty} \sum_{l=-k}^{j}\left(g_{l+1}-g_{l}\right)$ in $\mathcal{S}^{\prime}\left(\mathbb{R}^{n}\right)$.
\end{proof}
\par 
Next, to prove Theorem 6.2.
\begin{proof}
For each $j \in \mathbb{Z}$, consider the level set
$$
\mathcal{O}_{j} :=\left\{x \in \mathbb{R}^{n}: \mathcal{M} f(x)>2^{j}\right\}
$$
Then it follows immediately from the definition that
$$
\mathcal{O}_{j+1} \subset \mathcal{O}_{j} .
$$
Apply Lemma 2.2, then $f$ can be decomposed as
$$
f=g_{j}+b_{j}, \quad b_{j}=\sum_{k} b_{j, k}, \quad b_{j, k}=\left(f-c_{j, k}\right) \eta_{j, k}
$$
where each $b_{j, k}$ is supported in a cube $Q_{j, k}$ as described in Lemma 2.2.
$$
f=\sum_{j=-\infty}^{\infty}\left(g_{j+1}-g_{j}\right)
$$
with the series converging in the sense of distributions from Lemma 6.4. 
$$
f=\sum_{j, k} A_{j, k}, \quad g_{j+1}-g_{j}=\sum_{k} A_{j, k} \quad(j \in \mathbb{Z})
$$
in the sense of distributions, where each $A_{j, k}$, supported in $Q_{j, k}$, satisfies the pointwise estimate $\left|A_{j, k}(x)\right| \leq C_{0} 2^{j}$ for some universal constant $C_{0}$ and the moment condition $\int_{\mathbb{R}^{n}} A_{j, k}(x) q(x) d x=0$ for every $q(x) \in \mathcal{P}_{d}\left(\mathbb{R}^{n}\right)$. With these observations in mind, write
$$
a_{j, k} := \frac{A_{j, k}}{C_{0} 2^{j}}, \quad \kappa_{j, k} := C_{0} 2^{j} .
$$
Then we shall obtain that each $a_{j, k}$ satisfies
$$
\left|a_{j, k}\right| \leq \chi_{Q_{j, k}}, ~~~\quad~  \int_{\mathbb{R}^{n}} x^{\alpha} a_{j, k}(x) d x=0 
$$
and that $f=\sum_{j, k} \kappa_{j, k} a_{j, k}$ in the topology of $H L M_{\vec{p} \theta, w}\left(\mathbb{R}^{n}\right)$. Rearrange $\left\{a_{j, k}\right\}$ to obtain $\left\{a_{j}\right\}$. Do the same rearrangement for  $\left\{\lambda_{j, k}\right\}$.
To establish \eqref{79}, write
$$
\beta  := \left\|\left(\sum_{j=-\infty}^{\infty}\left|\lambda_{j} \chi_{Q_{j}}\right|^{v}\right)^{1 / v}\right\|_{L M_{\vec{p} \theta, w }} .
$$
Since
$$
\left\{\left(\kappa_{j, k} ; Q_{j, k}\right)\right\}_{j, k}=\left\{\left(\lambda_{j} ; Q_{j}\right)\right\}_{j},
$$
we have
$$
\beta =\left\|\left(\sum_{j=-\infty}^{\infty} \sum_{k \in K_{j}}\left|\kappa_{j, k} \chi_{Q_{j, k}}\right|^{v}\right)^{1 / v}\right\|_{L M_{\vec{p} \theta, w }} .
$$
By using the definition of $\kappa_{j}$, we then have
$$
\beta =C_{0}\left\|\left(\sum_{j=-\infty}^{\infty} \sum_{k \in K_{j}}\left|2^{j} \chi_{Q_{j, k}}\right|{ }^{v}\right)^{1 / v}\right\|_{L M_{\vec{p} \theta, w }}=C_{0}\left\|\left(\sum_{j=-\infty}^{\infty} 2^{j v} \sum_{k \in K_{j}} \chi_{Q_{j, k}}\right)^{1 / v}\right\|_{L M_{\vec{p} \theta, w }} .
$$
Observe that \eqref{88}, together with the bounded overlapping property, yields
$$
\chi_{\mathcal{O}_{j}}(x) \leq \sum_{k \in K_{j}} \chi_{Q_{j, k}}(x) \leq \sum_{k \in K_{j}} \chi_{200 Q_{j, k}}(x) \lesssim \chi_{\mathcal{O}_{j}}(x) \quad\left(x \in \mathbb{R}^{n}\right) .
$$
Thus, 
$$
\beta  \lesssim \left\|\left(\sum_{j=-\infty}^{\infty}\left(2^{j} \chi_{\mathcal{O}_{j}}\right)^{v}\right)^{1 / v}\right\|_{L M_{\vec{p} \theta, w }} .
$$
Recalling that $\mathcal{O}_{j} \supset \mathcal{O}_{j+1}$ for each $j \in \mathbb{Z}$, we have
$$
\sum_{j=-\infty}^{\infty}\left(2^{j} \chi_{\mathcal{O}_{j}}(x)\right)^{v} \sim\left(\sum_{j=-\infty}^{\infty} 2^{j} \chi_{\mathcal{O}_{j}}(x)\right)^{v} \sim\left(\sum_{j=-\infty}^{\infty} 2^{j} \chi_{\mathcal{O}_{j} \backslash \mathcal{O}_{j+1}}(x)\right)^{v} \quad\left(x \in \mathbb{R}^{n}\right) .
$$
Then, 
$$
\beta  \lesssim\left\|\sum_{j=-\infty}^{\infty} 2^{j} \chi_{\mathcal{O}_{j} \backslash \mathcal{O}_{j+1}}\right\|_{L M_{\vec{p} \theta, w }} .
$$
It follows by the definition of $\mathcal{O}_{j}$ that $2^{j}<\mathcal{M} f(x)$ for all $x \in \mathcal{O}_{j}$. Hence, we have
$$
\beta  \lesssim \left\|\sum_{j=-\infty}^{\infty} \chi_{\mathcal{O}_{j} \backslash \mathcal{O}_{j+1}} \mathcal{M} f\right\|_{L M_{\vec{p} \theta, w }} \lesssim\|\mathcal{M} f\|_{L M_{\vec{p} \theta, w }},
$$
So we receive the proof of Theorem 6.2.
\end{proof}
\section{The Hardy operators  on local mixed Morrey-type spaces}
\begin{theorem}
Suppose $1<\vec{p}<\infty, 1 \leq \theta \leq \infty, w \in \Omega_{\vec{p} \theta}$. Then
$
\|H f\|_{L M_{\vec{p} \theta, w}} \lesssim \|f\|_{L M_{\vec{p} \theta, w}}.
$
\end{theorem}
\begin{proof}
Let $f=\sum_{j=1}^{\infty} \lambda_{j} a_{j}$, $\mu$ stands for the Haar measure of $\mathrm{SO}(n)$\cite{GV2017}.
$$
S f(x) := \int_{\mathrm{SO}(n)} f(A x) d \mu(A)
$$
 Note that
 $$
 S: L M_{\vec{p} \theta, w}(\mathbb{R}) \rightarrow L M_{\vec{p} \theta, w}(\mathbb{R})
 $$
is a bounded linear opeator. Since
$$
\begin{aligned}
H f(x) & \sim \frac{1}{|x|^{n}} \int_{B(|x|)} f(y) d y \\
&=\int_{\operatorname{SO}(n)} \frac{1}{|A x|^{n}} \int_{B(|A x|)} f(y) d y d \mu(A) \\
&=\int_{\operatorname{SO}(n)} \frac{1}{|A x|^{n}} \int_{B(|A x|)} f(A y) d y d \mu(A) \\
&=\int_{\operatorname{SO}(n)} \frac{1}{|x|^{n}} \int_{B(|x|)} f(A y) d y d \mu(A) \\
&=H S f(x),
\end{aligned}
$$
therefore
$$
H f=H S f=\sum_{j=1}^{\infty} \lambda_{j} H S a_{j} .
$$
since $a_{j}$ is compactly supported 
$
\left|H S a_{j}\right| \lesssim S \chi_{Q_{j}},
$ and Theorem 6.2, 
$$
\begin{aligned}
\|H f\|_{L M_{\vec{p} \theta, w}} & \leq\left\|\sum_{j=1}^{\infty} \lambda_{j} H S a_{j}\right\|_{L M_{\vec{p} \theta, w}} 
 \lesssim \left\|\sum_{j=1}^{\infty} \lambda_{j} S_{\chi_{j}}\right\|_{L M_{\vec{p} \theta, w}} \\
& \lesssim \left\|\sum_{j=1}^{\infty} \lambda_{j} \chi_{Q_{j}}\right\|_{L M_{\vec{p} \theta, w}}
\lesssim \|f\|_{L M_{\vec{p} \theta, w}}
\end{aligned}
$$

\end{proof}
\section{Associate local mixed Morrey spaces with mixed Herz spaces}
Let $B_j=\{ x\in \mathbb{R}^n : |x| \leq 2^j\}$ and $A_j=B_j \setminus B_{j-1} $ for any $k \in \mathbb{Z}$. Denote $ \chi_j=\chi_{A_j}$, where $\chi_E$ is the charcteristic function of set E.
\begin{definition}(Homogeneous  Mixed Herz spaces) \cite{MW2021}
	Let $\alpha \in \mathbb{R}$, $0<p\leq \infty$, $0<\vec{q}\leq \infty$, where $\vec{q}=(q_1,q_2,\dots,q_n)$, then define  homogeneous mixed Herz spaces $\dot{K}_{\vec{q}}^{\alpha, p}(\mathbb{R}^n)$  by\\
	$$\dot{K}_{\vec{q}}^{\alpha, p}(\mathbb{R}^n):=\{f\in L_{loc}^{\vec{q}}: \|f\|_{\dot{K}_{\vec{q}}^{\alpha, p}}< \infty\},$$\\
	where$$\|f\|_{\dot{K}_{\vec{q}}^{\alpha, p}}=\left(\sum_{j=-\infty}^{\infty}2^{j \alpha p}\|f\chi_j\|_{L_{\vec{q}}}^p\right)^{\frac{1}{p}}.$$	
\end{definition}
\begin{lemma}
Let $ 1<\vec{p}\leq \infty, 1 \leq \theta <\infty $ and $0< \lambda<\sum_{i=1}^{n}\frac{1}{p_i}$ then for any $ w \subset \mathbb{R}^{n}$ 
$$
\|f\|_{L M^\lambda _{\vec{p} \theta }} \sim
\left(\sum_{j=-\infty}^{\infty}\left(2^{-\lambda j} \|f\|_{L_{\vec{p}} (B_j )} \right)^{\theta}\right)^{1 / \theta} 
$$
\end{lemma}
\begin{proof}
 Start with the proof of equality
$$	
	\begin{aligned}
		\|f\|_{L M^\lambda _{\vec{p} \theta }} &=\left(\int_{0}^{\infty}\left(t^{-\lambda} \|f\|_{L_{\vec{p}} (B_j )}\right)^{\theta} \frac{\mathrm{d} t}{t}\right)^{1 / \theta} \\
		&=\left(\sum_{j=-\infty}^{\infty} \int_{2^{j}}^{2^{j+1}}\left(t^{-\lambda} \|f\|_{L_{\vec{p}} (B_j )}\right)^{\theta} \frac{\mathrm{d} t}{t}\right)^{1 / \theta}.
	\end{aligned}
	$$
In fact
	$$
	\begin{array}{l}
	\left(\sum_{j=-\infty}^{\infty} \int_{2^{j}}^{2^{j+1}}\left(t^{-\lambda} \|f\|_{L_{\vec{p}} (B_j )}\right)^{\theta} \frac{\mathrm{d} t}{t}\right)^{1 / \theta} \\
 \leq 2^{\lambda}(\ln 2)^{1 / \theta}\left(\sum_{j=-\infty}^{\infty}\left(2^{-\lambda(j+1)}\|f\|_{L_{\vec{p}} (B_j )}\right)^{\theta}\right)^{1 / \theta}
	\end{array}
	$$
And
$$
\begin{array}{l}
	\left(\sum_{j=-\infty}^{\infty} \int_{2^{j}}^{2^{j+1}}\left(t^{-\lambda}\|f\|_{L_{\vec{p}} (B_j )} \right)^{\theta} \frac{\mathrm{d} t}{t}\right)^{1 / \theta} \\
	\geq 2^{-\lambda}(\ln 2)^{1 / \theta}\left(\sum_{j=-\infty}^{\infty}\left(2^{-\lambda j} \|f\|_{L_{\vec{p}} (B_j )}\right)^{\theta}\right)^{1 / \theta}
\end{array}
$$

\end{proof}
\begin{theorem}
	Let $1<p<\infty, 1 \leq \theta \leq \infty$ , $0<\lambda<\sum_{i=1}^{n}\frac{1}{p_i}$ and 	for all measurable functions $f: \mathbb{R}^{n} \rightarrow \mathbb{C}$.
	Then
	$$
	\|f\|_{L M_{\vec{p} \theta}^{\lambda}} \sim\|f\|_{\dot{K}_{\vec{q}}^{\lambda, \theta}}.
	$$

\end{theorem}

\begin{proof}
 It is clear from that
$$
\|f\|_{L M_{\vec{p} \theta}^{\lambda}} \gtrsim\left\{\sum_{j=-\infty}^{\infty}\left(2^{-\lambda j}\|f\chi_j\|_{L_{\vec{q}}}\right)^{\theta}\right\}^{\frac{1}{\theta}}.
$$
To prove the reverse estimate, 
$$\begin{aligned}
\|f\|_{L M_{\vec{p} \theta}^{\lambda}}& \sim\left(\sum_{j=-\infty}^{\infty}\left(2^{-\lambda j} \|f\|_{L_{\vec{p}} (B_j )} \right)^{\theta}\right)^{1 / \theta} \\
&=\left\{\sum_{j=-\infty}^{\infty}\left(\sum_{k=-\infty}^{j} 2^{-\lambda j}\|f\chi_j\|_{L_{\vec{q}}}\right)^{\theta}\right\}^{\frac{1}{\theta}}\\
	&=\left\{\sum_{j=-\infty}^{\infty}\left(\sum_{k=-\infty}^{\infty} \chi_{(-\infty, j]}(k) 2^{-\lambda j}\|f\chi_j\|_{L_{\vec{q}}}\right)^{\theta}\right\}^{\frac{1}{\theta}}\\
	\end{aligned}$$
	$$\begin{aligned}
	~~~~\quad~~~~~
	&\leq \sum_{k=-\infty}^{\infty}\left\{\sum_{j=-\infty}^{\infty}\left(\chi_{(-\infty, j]}(k) 2^{-\lambda j}\|f\chi_j\|_{L_{\vec{q}}}\right)^{\theta}\right\}^{\frac{1}{\theta}}\\
	&=\left\{\sum_{k=-\infty}^{\infty}\left(\frac{1}{1-2^{-\lambda}} \cdot 2^{-\lambda k}\|f\chi_j\|_{L_{\vec{q}}}\right)^{\theta}\right\}^{\frac{1}{\theta}}.
\end{aligned}
$$
 The proof of Theorem 8.1  is complete.
\end{proof}

\hspace*{-0.6cm}\textbf{\bf Competing interests}\\
	The authors declare that they have no competing interests.\\
	
\hspace*{-0.6cm}\textbf{\bf Funding}\\
	The research was supported by Natural Science Foundation of China (Grant Nos. 12061069).\\
	
\hspace*{-0.6cm}\textbf{\bf Authors contributions}\\
	All authors contributed equality and significantly in writing this paper. All authors read and approved the final manuscript.\\
	
\hspace*{-0.6cm}\textbf{\bf Acknowledgments}\\
	The authors would like to express their thanks to the referees for valuable advice regarding previous version of this paper.\\
	
	\hspace*{-0.6cm}\textbf{\bf Authors detaials}\\
	Mingwei shi and Jiang Zhou*, moluxiangfeng888@163.com and zhoujiang@xju.edu.cn, College of Mathematics and System Science, Xinjiang University, Urumqi, 830046, P.R China.

\bigskip
\noindent Mingwei shi\\
\medskip
\noindent
College of Mathematics and System Sciences\\
Xinjiang University\\
Urumqi 830046\\
\smallskip
\noindent{e-mail }:
\texttt{moluxiangfeng888@163.com} (Mingwei shi)

\bigskip
\noindent Jiang Zhou\\
\medskip
\noindent
College of Mathematics and System Sciences\\
Xinjiang University\\
Urumqi 830046\\
\smallskip
\noindent{e-mail }:
\texttt{zhoujiang@xju.edu.cn} (Mingwei shi)
\bigskip \medskip

\vskip 0.5cm
\begin{thebibliography}{999}
\bibitem{CB1938}C.B. Morrey, On the solutions of quasi-linear elliptic partial differential equations, Trans. Amer. Math. Soc. 43 (1938), no. 1, 126–166.
\bibitem{DR1975} D.R. Adams, A note on Riesz potentials, Duke Math. J. 42 (1975), 765–778.	
\bibitem{BV2004}V. Burenkov  and H. V. Guliyev, Necessary and sufficient conditions for boundedness
of the maximal operator in the local Morrey-type spaces, Studia Math. 163(2):157-176(2004)

\bibitem{BV2006} V.I. Burenkov, V.S. Guliev, G.V. Guliev,  Necessary and sufficient conditions for the boundedness of the fractional maximal operator in local Morrey-type spaces. Dokl. Math. 74, 540–544 (2006). 



\bibitem{BJ2011}V.I. Burenkov,  P. Jain, T.V. Tararykova, On boundedness of the Hardy operator inMorrey-type spaces,
Eurasian Math. J. 2(1), 52–80 (2011)
\bibitem{GV2017} V.S. Guliyev, G. H.  Sabir, Y. Sawano, Decompositions of local Morrey-type spaces, Positivity 21.3 (2017): 1223-1252.

\bibitem{AB1961}A.   Benedek, R. Panzone, The space $L^{\vec{p}}$, with mixed norm, Duke Mathematical Journal, 1961, 28(3):301-324.

\bibitem{cs2017} T.  Chen, W.Sun, Iterated and Mixed Weak Norms with Applications to Geometric Inequalities[J], arXiv, 2017.

\bibitem{cg2017} G.  Cleanthous, A. G. Georgiadis,  M. Nielsen, Molecular decomposition of anisotropic homogeneous mixed-norm spaces with applications to the boundedness of operators[J], Applied and Computational Harmonic Analysis, 2017, 47(2):447-480.

\bibitem{FD1977} D. L. Fernandez, Lorentz spaces, with mixed norms[J]. Journal of Functional Analysis, 1977, 25(2):128-146.

\bibitem{BA1981}A. P. Blozinski, Multivariate rearrangements and Banach function spaces with mixed norms[J], Transactions of the American Mathematical Society, 1981, 263(1):149-149.

\bibitem{NT2019}T. Nogayama,  Mixed Morrey spaces, Positivity ,2019. 	
\bibitem{ZH2021} H  Zhang, J Zhou, The Boundedness of Fractional Integral Operators in Local and Global Mixed Morrey-type Spaces. Positivity, 2021.



\bibitem{SJ2022} M. W. Shi, J Zhou, The Hardy-Littewood maximal operator in Local and
Global mixed Morrey-type Spaces , Submitted

\bibitem{SE1993}E. M. Stein, HarmonicAnalysis, real-variable methods, orthogonality, and oscillatory integrals. Princeton
University Press, Princeton (1993).
\bibitem{MW2021} M. Q. Wei,  A characterization of $C\dot{M}O^{\vec{q}}$ via the commutator of Hardy-typeoperators on mixed Herz spaces, ~ Applicable Analysis, 2021.



\bibitem{BS2014}T. Batbold, Y. Sawano, Decompositions for local Morrey spaces, Eurasian Math. J. 5(3), 2014.
\end{thebibliography}
\end{document}